\documentclass{amsart}
\usepackage[latin1]{inputenc}
\usepackage{amssymb}
\usepackage{latexsym}
\usepackage{amscd}
\usepackage{longtable}
\usepackage{amsmath}
\usepackage{mathrsfs}
\usepackage{array}
\usepackage[all]{xy}
\usepackage{a4}

\newtheorem{defn0}{Definition}[section]
\newtheorem{prop0}[defn0]{Proposition}
\newtheorem{thm0}[defn0]{Theorem}
\newtheorem{lemma0}[defn0]{Lemma}
\newtheorem{claim0}[defn0]{Claim}
\newtheorem{corollary0}[defn0]{Corollary}
\newtheorem{example0}[defn0]{Example}
\newtheorem{remark0}[defn0]{Remark}
\newtheorem{assumption0}[defn0]{Assumption}
\newtheorem{conjecture0}[defn0]{Conjecture}
\newtheorem{notation0}[defn0]{Notation}
\newtheorem{question0}[defn0]{Question}

\newenvironment{definition}{\begin{defn0}\rm}{\end{defn0}}
\newenvironment{proposition}{\begin{prop0}}{\end{prop0}}

\newenvironment{lemma}{\begin{lemma0}}{\end{lemma0}}

\newenvironment{remark}{\begin{remark0}\rm}{\end{remark0}}

\newenvironment{conjecture}{\begin{conjecture0}}{\end{conjecture0}}

\def \mint {\times \hskip -1.1em \int}

\newcommand{\Gal}{{\mathrm {Gal}}}

\newcommand{\ord}{\mathrm{ord}}

\newcommand{\Jac}{\mathrm{Jac}}

\newcommand{\PGL}{{\mathrm{PGL}}}

\newcommand{\GL}{{\mathrm{GL}}}

\newcommand{\Z}{{\mathbb Z}}
\newcommand{\A}{{\mathbb A}}
\newcommand{\Q}{{\mathbb Q}}
\newcommand{\C}{{\mathbb C}}
\newcommand{\R}{{\mathbb R}}

\newcommand{\aH}{{\mathbb H}}

\newcommand{\PP}{{\mathbb P}}

\newcommand{\cA}{{\mathcal A}}

\newcommand{\cC}{{\mathcal C}}
\newcommand{\cF}{{\mathcal F}}

\newcommand{\cH}{{\mathcal H}}

\newcommand{\cG}{{\mathcal G}}
\newcommand{\cV}{{\mathcal V}}

\newcommand{\cU}{{\mathcal U}}
\newcommand{\cE}{{\mathcal E}}

\newcommand{\cO}{{\mathcal O}}

\newcommand{\Div}{{\mathrm {Div}}}
\newcommand{\Coind}{{\mathrm {Coind}}}

\newcommand{\ra}{{\rightarrow}}
\newcommand{\lra}{\longrightarrow}
\newcommand{\Hom}{{\mathrm {Hom}}}

\newcommand{\SigmaB}{\Sigma_B}

\newcommand{\St}{\mathrm{St}}
\newcommand{\ev}{\operatorname{ev}}

\newcommand{\smtx}[4]{\left(\begin{smallmatrix}#1&#2\\#3&#4\end{smallmatrix}\right)}

\def\Xint#1{\mathchoice
{\XXint\displaystyle\textstyle{#1}}%
{\XXint\textstyle\scriptstyle{#1}}%
{\XXint\scriptstyle\scriptscriptstyle{#1}}%
{\XXint\scriptscriptstyle\scriptscriptstyle{#1}}%
\!\int}
\def\XXint#1#2#3{{\setbox0=\hbox{$#1{#2#3}{\int}$}
\vcenter{\hbox{$#2#3$}}\kern-.5\wd0}}

\renewcommand{\setminus}{\smallsetminus}
\newcommand{\ns}{\text{ns}}

\title{An automorphic approach to Darmon points}

\author{Xavier Guitart}
\address{Departament de Matematiques i Informatica\\
Universitat de Barcelona\\
Catalunya
}
\email{xevi.guitart@gmail.com}
\thanks{}

\author{Marc Masdeu}
\address{Departament de Matematiques \\
Universitat Autonoma de Barcelona \\
Catalunya}
\email{masdeu@mat.uab.cat}
\thanks{}

\author{Santiago Molina}
\address{Centre de Recerca Matematica\\
Campus de Bellaterra, Edifici C\\
Catalunya
}
\email{smolina@crm.cat}
\urladdr{}
\thanks{}

\begin{document}
%%%%%%%%%%%%%%%%%%%%%%%%%%%%%%%%%%%%%%%%%%%%%

\maketitle
%%%%%%%%%%%%%%%%%%%%%%%%%%%%%%%%%%%%%%%%%%%%%
%\vskip 1cm

\begin{abstract}
  We give archimedean and non-archimedean constructions of Darmon points on modular abelian varieties attached to automorphic forms over arbitrary number fields and possibly non-trivial central character. An effort is made to present a unifying point of view, emphasizing the automorphic nature of the construction.
\end{abstract}

\section{Introduction}
Let $F$ be a number field and let $\Pi$ be a cuspidal automorphic representation of $\GL_2(\A_F)$ of parallel weight $2$. We denote by $\omega_\Pi$ the central character of $\Pi$, by $L_\Pi$ the field of definition, and by $N$ the level of the associated newform. In general it is expected that there exists a simple abelian variety $A_\Pi/F$ of dimension $[L_\Pi\colon \Q]$ determined, up to isogeny, by the equality of $L$-series
\begin{align*}
  L(A_\Pi, s) = \prod_{\Pi^\sigma}L(\Pi^\sigma,s-1/2),
\end{align*}
where $\Pi^\sigma$ runs over the Galois conjugates of $\Pi$ (cf. \cite[Conjecture 3]{Taylor}). This is indeed known in many circumstances when $F$ is totally real. If $F$ is totally complex, \cite[Conjecture 3]{Taylor} also allows the possibility that $A_\Pi$ has dimension $2[L_\Pi\colon \Q]$; as we will see in Remark \ref{remark:A_Pi} below, throughout the article we will assume certain conditions which, together with the expected properties of $A_\Pi$, rule out this possibility. Let also $K/F$ be a quadratic extension. If $F$ is totally real and $K$ is CM, the theory of Heegner points provides a canonical construction of algebraic points on $A_\Pi$, which are defined over abelian extensions of $K$ and whose height is controlled by the formulas of Gross--Zagier and Yuan--Zhang--Zhang. 

Darmon points are a collection of conjectural generalizations of the Heegner point construction to the case where $K/F$ is a quadratic extension which is not CM. In all of these constructions, the definition of the points is of a local nature. Namely, one fixes a place $v$ of $F$, which can be archimedean or non-archimedean, with the property that it does not split in $K$. One needs to assume as well that $A_\Pi$ admits a $v$-adic uniformization, which is automatic  if $v$ is archimedean. The points are then constructed as local points living (at least conjecturally) in $A_\Pi(K_v)$, by a method that resembles and generalizes the local formulas for Heegner points as (complex or $p$-adic) integrals. The resulting local points are then conjectured to actually lie on $A_\Pi(K^{\textrm{ab}})$ and to satisfy properties akin to those of Heegner points, such as Shimura's reciprocity law or a formula of Gross--Zagier type.

The original and seminal constructions, both non-archimedean and archimedean, were introduced by Darmon under certain additional restrictions on the fields and automorphic forms considered.  The $p$-adic construction of \cite{Dar-int} treats the case where $F=\Q$, the quadratic extension $K/\Q$ is real, $\dim (A_\Pi)=1$, and $N$ satisfies an analogue of the Heegner condition with respect to $K$. The archimedean construction of \cite[Chapter 9]{Dar-rpmec} assumes that $F$ is real quadratic of class number one, $K$ has exactly one complex place, and $N$ satisfies a Heegner condition with respect to $K$.

Over the years, several authors have subsequently relaxed these restrictions by adapting Darmon's original constructions to more general settings. Most relevantly to the framework of this article, Dasgupta \cite{Dasgupta} lifted Darmon's $p$-adic points to modular Jacobians, Greenberg \cite{Gr} generalized the $p$-adic construction to the setting where $F$ is totally real of narrow class number one and $K/F$ is any quadratic extension (and then $N$ satisfies a generalized Heegner condition);  Gärtner \cite{Ga-art} relaxed the restrictions in the archimedean case by allowing an arbitrary totally real base field $F$ (with no restrictions on the narrow class number) and $K/F$ any quadratic extension (again with $N$ satisfying a generalized Heegner condition);  these constructions were extended in \cite{GMS} to the case where $F$ is of arbitrary signature and of narrow class number one.  While in all of these cases the central character $\omega_\Pi$ was assumed to be trivial (and so the relevant congruence subgroup is of the type $\Gamma_0$) the constructions of Kohen--Pacetti and of Masdeu appearing in~\cite{KPM} use the non-split Cartan subgroup to define Darmon points on elliptic curves with non-squarefree conductor. Finally, it is also worth pointing out that Darmon's constructions have been generalized in other directions; for example,  Rotger--Seveso \cite{Rotger-Seveso} deal with cycles in motives attached to higher weight modular forms.

The aim of the present note is to provide a general construction of Darmon-like points, both archimedean and non-archimedean, which works for any base field $F$ (with no restrictions on the narrow class number), arbitrary quadratic extensions $K/F$ (with $N$ satisfying a generalized Heegner condition), and non-trivial central characters $\omega_\Pi$. In particular, we provide the first constructions of Darmon points which allow an arbitrary central character, and in the $p$-adic case also the first one that allows base fields of narrow class number greater than one.

In dealing with  base fields of arbitrary class number, the adelic formulation of automorphic forms is more convenient. This was already the case in Gartner's work, but in our constructions we rather follow the group cohomology-based formalism and techniques introduced by Spiess \cite{Spiess} in his study of automorphic $L$-invariants. 

In fact, such automorphic $L$-invariants already play a role in the non-archimedean $v$-adic Darmon points, and they were also initially introduced in \cite{Dar-int}. The relation stems from the fact that Darmon points are initially defined as points on a certain $v$-adic torus, whose isogeny class is given by an automorphic $L$-invariant of $\Pi$. It is generally expected, and in some cases proven (cf. \cite{lennart-L-invariants} for the latest and more general results in this direction), that this torus is isogenous to the $v$-adic uniformization of $A_\Pi$, and this is what allows the points to be regarded as lying in $A_\Pi(K_v)$. 

The versatility and generality of Spiess's approach is thus well suited for giving a non-archimedean construction of Darmon points in the above-mentioned generality of arbitrary class number base fields and arbitrary central characters. Our construction of archimedean Darmon points is also inspired by Spiess's point of view, and it can be seen as an adaptation of his methods to the case where the distinguished local place $v$ is archimedean.

In Section \ref{section: automorphic forms}, in addition to presenting the relevant notions and notations from automorphic forms, we introduce two maps that will play a key role in the construction of the Darmon points of this article: the connection and evaluation morphisms, which we present first in the archimedean setting and then in the non-archimedean setting. In Section \ref{section: oda} we use these morphisms to construct certain tori (both complex and $p$-adic) which are conjectured to be isogenous to the corresponding local uniformizations of $A_\Pi$. It is in Section \ref{D-P} where we give the construction of Darmon points which, taking advantage of the adelic language, we give in a unified presentation that encompasses both the archimedean and non-archimedean cases. Finally, in the Appendix we provide a detailed account of the general construction in the particular case where $K/F$ is CM, with the aim of checking that our construction does particularize to the known case of Heegner points in the CM setting; this can also be helpful to the reader, as it provides motivation for some of the general constructions made throughout the article. 

\subsection*{Acknowledgments} We thank Lennart Gehrmann for helpful discussions during the course of this work. Guitart was partially supported by projects MTM2015-66716-P and MTM2015-63829-P. Masdeu was partially supported by project MSC--IF--H2020--ExplicitDarmonProg. This project has received funding from the European Research Council (ERC) under the European Union's Horizon 2020 research and innovation programme (grant agreement No 682152).

\section{Automorphic Forms and Group Cohomology}\label{section: automorphic forms}

In this section we introduce the notation that will be used in the rest of the note. We show how the automorphic representations of interest arise in certain cohomological degrees. This will be crucially used in the remaining sections. Although that notationally the construction is the same in the archimedean and non-archimedean cases, the details have to be worked out separately, as can be seen in the different subsections.

\subsection{Set up} Let $F$ be an arbitrary number field and let $\A_F$ be the adeles of $F$. % and $\I_F=\I_F$ the ideles of $F$.
We also denote by $\infty$ the set of archimedean places of $F$.
Let $B$ be a quaternion algebra over $F$ and $G$ be the algebraic group $B^\times$. Let $\SigmaB$ be the set of archimedean places where $B$ splits.  Write $G(F)^+\subset G(F)$ for the subgroup of elements of totally positive norm, namely, the subset of elements in $G(F)$ whose image through the natural embedding in $G(F_\sigma)$, for any real place $\sigma\in\Sigma_B$, has positive norm. For every $\sigma\in\Sigma_B$ we fix $K_\sigma$, a maximal compact subgroup of $G(F_\sigma)$.% isomorphic to $\rO(2)$.

Let $S$ be a finite set of places containing $\SigmaB$. We put
$$F_S=\prod_{v\in S}F_v,  \qquad  \A_F^{S}=\A_F\cap \prod_{v\not\in S}F_v.$$ %\ \ \text{ and }\ \  \I_F^{S,\infty}=\I_F\cap \prod_{v\not\in S\cup\infty}F_v^\times.$$  %we define $\cM_T$ to be the free group generated by elements of $G(F)\slash T^2(F)$.
Given $M$ a $G(F)$-module over a ring $R$ and an $R$-module $N$ we define
\begin{align*}
 \cA^S(M,N) = \{ & f:G(\A_F^{S})\longrightarrow \Hom_R(M,N) \colon \text{ there exists} \\
& \text{ an open compact subgroup } U\subseteq G(\A_F^{S}) \text{ with } f(\cdot\; U)=f(\cdot)\}.
\end{align*}
For $g\in G(F)$ we will denote by  $g^S$ the image of $g$ under the natural projection   $G(F)\rightarrow G(\A_F^{S})$. Note that $\cA^S(M,N)$ is equipped with commuting $G(F)$ and $G(\A_F^{S})$-actions: if $g\in G(\A_F^{S})$ and $f\in \cA^S(M,N)$ then
\[
\begin{array}{cc}
(h\cdot f)(g)=h(f((h^{-1})^S g)),&h\in G(F), \\
(h\cdot f)(g)=f(g h)),&h\in G(\A_F^{S}).\\
\end{array}
\]
If $\varphi\colon N_1\ra N_2$ is a morphism then by abuse of notation we denote also by $\varphi$ the map induced in cohomology
\begin{align*}
\varphi\colon  H^n(G(F), \cA^S(M,N_1)) \lra   H^n(G(F), \cA^S(M,N_2)).
\end{align*}

% We write $\cA_f^v(M,N)$ instead of $\cA_f^{\{v\}}(M,N)$ and $\cA_f(M,N)$ instead of $\cA^\emptyset(M,N)$.
Write $\cA^S(N):=\cA^S(R,N)$.
If $N$ is a $L$-vector space for some field $L$, then $H^n(H,\cA^S(N))$ is a smooth $G(\A_F^{S})$-representation over $L$, for any subgroup $H\subseteq G(F)$.
\begin{remark}\label{AvsHom}
Let $S'\supseteq S$ be a finite set of places and write $T=S'\setminus S$.
For any $G(F)$-module $M$ and any $G(F_T)$-representation $V$ over $R$ we have an isomorphism of $(G(F),G(\A_F^{S'}))$-representations:
\begin{align*}
\phi:\Hom_{G(F_T)}(V,\cA^S(M,N)) \longrightarrow \cA^{S'}( V\otimes_R M,N).
\end{align*}
\end{remark}

We define $\cA(\C)$ to be the $\C$-vector space of functions $f:G(\A_F)\longrightarrow\C$ which satisfy:
\begin{itemize}
\item There is an open compact subgroup $U_f\subseteq G(\A_F^\infty)$ such that $f(\cdot\; U_f)=f(\cdot)$.

\item $f\mid_{G(F_\infty)}$ is $\cC^\infty$, where $F_\infty$ is the product of the completions of $F$ at all the archimedean places.

\item For all $\sigma\in\Sigma_B$, $f\in\cA(\C)$ is $K_\sigma$-finite, namely, its right translates by elements of $K_\sigma$ span a finite-dimensional vector space.

\item For all $\sigma\in\Sigma_B$, $f$ is ${\mathcal{Z}_\sigma}$-finite, where ${\mathcal{Z}_\sigma}$ is the centre of the universal enveloping algebra of $G(F_\sigma)$.
\end{itemize}
The action of $G(\A_F)$ on $\cA(\C)$  by right translation defines a smooth $G(\A_F^\infty)$-representation. For an archimedean place $\sigma\in\Sigma_B$ it is a $(\cG_\sigma,K_\sigma)$-module, where $\cG_\sigma$ is the Lie algebra of $G(F_\sigma)$. Moreover, $\cA(\C)$ is also equipped with a $G(F)$-action:
\[
(h\cdot f)(g)=f(h^{-1} g),\qquad h\in G(F),
\]
where $ g\in G(\A_F)$ and $f\in \cA(\C)$.
\begin{remark}
  Note that automorphic forms can be seen in a natural way as elements of $H^0(G(F),\cA(\C))$. However, in this note we will also be interested in the realization of automorphic representations in certain higher cohomology groups $H^i(G(F),\cA(\C))$.
\end{remark}

Let $\cV$ be a product of $(\cG_\sigma,K_\sigma)$-modules for all archimedean $\sigma\in\SigmaB$ and $G(F_v)$-representations at all $v\in S\setminus\SigmaB$. %It is a $\prod_\sigma (\cG_\sigma, K_\sigma)$-module; as a shorthand, we will say that it is a $(\cG_\infty,K_\infty)$-module.
In analogy with Remark \ref{AvsHom}, we define
\[
\cA^S( \cV,\C):=\Hom(\cV,\cA(\C)),
\]
where the homomorphisms are to be understood in the cathegory of $\prod_{\sigma\in\SigmaB}(\cG_\sigma,K_\sigma)\times G(F_{S\setminus\SigmaB})$-modules. Note that $\cA^S( V,\C)$ in naturally endowed with actions of $G(F)$ and $G(\A_F^S)$. We say that a $G(F_S)$-representation $V$ is finite at $\Sigma_B$ if, for any $\sigma\in\Sigma_B$, the $G(F_\sigma)$-representation $V$ is finite dimensional. Note that in this situation $V$ is simultaneously a $G(F_S)$-representation and a $\prod_{\sigma\in\SigmaB}(\cG_\sigma,K_\sigma)\times G(F_{S\setminus\SigmaB})$-module (that will be denoted by $\cV$). The following lemma ensures that this notation is consistent with the previous one.
\begin{lemma}
  Let $V$ be a $G(F_S)$-representation finite at $\Sigma_B$ and let $\cV$ denote the corresponding $\prod_{\sigma\in\SigmaB}(\cG_\sigma,K_\sigma)\times G(F_{S\setminus\SigmaB})$-module. Then
  \[
  \cA^S(\cV,\C)=\cA^S(V,\C).
  \]
\end{lemma}
\begin{proof}
Similarly as in Remark \ref{AvsHom}, it is easy to check that
\[
\cA^S(V,\C)=\Hom_{G(F^S)}(V,\cA(\C))\hookrightarrow \Hom(\cV,\cA(\C))= \cA^S(\cV,\C).
\]
On the other side, for any $\phi\in \Hom(\cV,\cA(\C))$, $h\in G(\A^{\Sigma_B})$ and $v\in V$, we define
\[
h^{v,h}_{\phi}:G(F_{\Sigma_B})\rightarrow\C,\qquad h^{v,h}_\phi(g)=\phi(g^{-1}v)(g,h).
\]
Since $\phi$ is a $(\cG_\sigma,K_\sigma)$-module morphism, for $\sigma\in\Sigma_B$, the function $h_{\phi}^{v,h}$ is $K_\sigma$-invariant, hence it defines a $\cC^\infty$ function on $\cH^{r_1}\times\aH^{r_2}$, where $\cH=\{z\in\C,\;{\rm Im}(z)>0\}$ is the Poincar\'e upper half-plane and $\aH=\{(x,y)\in\C\times\R,\;y>0\}$ is the upper half space. Moreover, the Lie algebra $\cG_\sigma$ acts trivially on $h_{\phi}^{v,h}$, hence $h_{\phi}^{v,h}$ is constant and $\phi\in \Hom_{G(F^S)}(V,\cA(\C))$.
\end{proof}

%If $S$ is a finite set of places containing all the archimedean places, and $V$ is a $G(F_S)$-representation (meaning that it is a product of a $(\cG_\infty,K_\infty)$-representation and $G(F_v)$-representations for the finite places $v\in S$), we define
%\[
%\cA^S(V,\C):=\Hom_{G(F_S)}(V,\cA(\C)).
%\]
%Notice that if $V$ is the trivial $(\cG_\sigma,K_\sigma)$-module at each infinite $\sigma$, then $\cA^S(V,\C)=\cA_f^{S\setminus \infty}(V,\C)$.

Given an automorphic representation $\pi$ of $G$ of weight $2$, we denote by $L_\pi$ its \emph{field of definition}, namely, the minimal field with the property that there exists an irreducible representation $\rho$ of $G(\A_F^\infty)$ over $L_\pi$ such that $\rho\otimes_{L_\pi}\C\simeq\pi\mid_{G(\A_F^\infty)}$. Observe that, since for any archimedean place not in $\SigmaB$ the representation $\pi_\sigma$ is trivial, the representation $\rho$ can be seen as a representation of $G(\A_F^{\SigmaB})$ over $L_\pi$.
For any $S$ as above, and any $G(\A_F^{S})$-representation $W$ over an extension $k$ of $L_\pi$, we write
\[
W_\pi:=\Hom_{G(\A_F^{S})}(\rho\mid_{G(\A_F^{S})}\otimes_{L_\pi}k,W).
\]
We denote by $W(\pi)\subseteq W$ the isotypical component of $\pi$; that is, the image of the elements of $W_\pi$.

\subsection{Archimedean connection morphisms}

Consider $\sigma\in \SigmaB$, a real archimedean place at which $B$ splits. Let $D_\sigma$ be the $(\cG_\sigma,K_\sigma)$-module of discrete series of weight $2$. % if $\sigma$ splits in $B$, and $\C$ if $\sigma$ ramifies.
Let $I_\sigma^+$ and $I_\sigma^-$ be the $(\cG_\sigma,K_\sigma)$-modules defined in \cite[Appendix 2]{Santi}, which sit in exact sequences:
\begin{equation}\label{eq: successio I+}
0\longrightarrow D_\sigma\longrightarrow  I_\sigma^{\pm}\longrightarrow\C(\pm 1)\longrightarrow 0;
\end{equation}
where $\C(-1)$ is $\C$ on which a matrix of $G(F_\sigma)$ acts by the sign of its determinant, and $\C(+1)=\C$ with trivial action.

Choose a sign at each real place in $\SigmaB$; equivalently choose a character
\begin{equation}\label{eq:lambda}
\lambda\colon G(F_{\Sigma_B})/G(F_{\SigmaB})^+\rightarrow \{\pm 1\}.
\end{equation}
As above, we denote by $S$ a finite set of places containing $\Sigma_B$. Let $V$ be a finite dimensional $G(F_{S\setminus\{\sigma\}})$-representation with associated  $\prod_{\tau\in\SigmaB\setminus\{\sigma\}}(\cG_\tau,K_\tau)\times G(F_{S\setminus\SigmaB})$-module $\cV$.
Then \eqref{eq: successio I+} gives rise to an exact sequence
\begin{eqnarray*}
0\longrightarrow \cA^S(V\otimes \C(\lambda(\sigma)),\C)\stackrel{\iota_\lambda}{\longrightarrow} \cA^S(\cV\otimes I_\sigma^{\lambda(\sigma)},\C)\stackrel{{\rm pr}_\lambda}{\longrightarrow}\cA^S(\cV\otimes D_\sigma,\C)\longrightarrow 0.\label{exseqfund1}
\end{eqnarray*}
which induces, for each $q\geq 0$, connection homomorphisms
\begin{eqnarray*}
\delta_\sigma^{\lambda(\sigma)}\colon H^q(G(F), \cA^S(\cV\otimes D_\sigma,\C)) \longrightarrow  H^{q+1}(G(F),\cA^S(V\otimes\C(\lambda(\sigma)),\C)).
\end{eqnarray*}
%Now let $S_f=S\setminus \infty$ and $V$ a $G(F_{S_f})$-representation. By iterating the above connection homomorphisms we obtain
%\begin{align*}
%  \delta^{\lambda}\colon H^0(G(F), \cA^S(V\otimes \prod_{\sigma\mid \infty} D_\sigma,\C)) \longrightarrow  H^{r}(G(F),\cA^{S_f}_f(V\otimes \prod_{\sigma\mid \infty} \C(\lambda(\sigma),\C)).
%\end{align*}
%Note that we can identify
%\[
%H^{r}(G(F)^+,\cA^{S_f}_f(V,\C))^\lambda:=H^{r}(G(F),\cA^{S_f}_f(V\otimes \prod_{\sigma\mid \infty} \C(\lambda(\sigma),\C)),
%\]
%with the subspace of $H^{r}(G(F)^+,\cA^{S_f}_f(V,\C))$ where $G(F)/G(F)^+$ acts by the character $G(F)/G(F)^+\stackrel{\det}{\ra}F/F^+\stackrel{\lambda}{\ra}\pm 1$.

\subsection{Archimedean evaluation morphisms} 

Fix $f_\sigma^\pm$ to be the weight $0$ generator of $I_\sigma^\pm$ (see \cite[\S 1]{molina} for more details) mapping to $1$ under \eqref{eq: successio I+},  %Let $S$ be a finite set of places, containing all the archimedean ones and let $S_f=S\setminus \infty$ be the finite places in $S$, and let $S'=S\setminus\{\sigma\}$. 
and let $V$ and $\cV$ be as above. Any element of weight $0$ in $I_\sigma^\pm$ can be evaluated at elements in $\Delta_\sigma =\Div( \cH)$; this induces an evaluation morphism of $G(F)^+$-modules
\begin{align*}
  \ev_\sigma^\pm \colon \cA^S(\cV\otimes I_\sigma^\pm,\C)\lra \cA^{S}(V\otimes \Delta_\sigma,\C)
\end{align*}
by means of the formula
\begin{align*}
  \ev_\sigma^\pm (\varphi)(g)(v\otimes z) = \varphi(v\otimes f_\sigma^\pm)( z, g), \ \ \ \text{where }  z\in\cH,\;  g\in G(\A^S_F),\; v\in V.
\end{align*}
Write $\Delta_\sigma^0=\Div^0(\cH)$, and note that the above evaluation map induces a morphism of $G(F)^+$-modules 
\[
ev_\sigma^\pm\colon\cA^{S}(\cV\otimes D_\sigma,\C)\ra \cA^{S}(V\otimes \Delta_\sigma^0,\C)
\]
which makes the following diagram of $G(F)^+$-modules commutative:
\begin{align}
  \label{eq:evaluation-archimedean}
\xymatrix{
0\ar[r]&\cA^S(V\otimes \C,\C) \ar[r]\ar[d]^{\mathrm{id}}&\cA^{S}(\cV\otimes I_\sigma^\pm,\C)\ar[r]\ar[d]^{\ev_\sigma^\pm}& \cA^{S}(\cV\otimes D_\sigma,\C) \ar[r]\ar@{-->}[d]^{\ev_\sigma^\pm}&0\\
0\ar[r]& \cA^{S}(V,\C) \ar[r]&\cA^{S}(V\otimes \Delta_\sigma,\C)\ar[r]&\cA^{S}(V\otimes \Delta_\sigma^0,\C)\ar[r]&0.
}
\end{align}
 Here, the bottom exact sequence is induced by the exact sequence defining the degree $0$ divisors $0\ra \Delta^0_\sigma \ra \Delta_\sigma\ra \Z\ra 0.$
\subsection{Non-archimedean connection morphisms}
From now on, we assume that $B$ is split at all places in $S$.
Fix a non-archimedean place $v\in S$, and let $\ell\colon F_v^\times\ra M$ be a continuous homomorphism of topological abelian groups, which we write additively. We denote by $\St_v^M$ be the Steinberg representation of $G(F_v)$ over $M$; that is to say, $\St_v^M = C(\PP^1(F_v),M)/M$ (the $M$-valued continuous functions of $\PP^1(F_v)$ modulo constants). Following the ideas in~\cite{lennart-felix} we define
\begin{align*}
  \cE(\ell) = \{ (\phi,y)\in C(G(F_v),M)\times \Z \colon \phi\left( \smtx s x o t g \right)= \phi(g) + y \ell(t)  \}/(M,0).
\end{align*}
Then the projection onto the second component induces an exact sequence
\begin{eqnarray}\label{eq: successio st}
0\longrightarrow \St_v^M\longrightarrow  \cE(\ell)\longrightarrow \Z\longrightarrow 0.
\end{eqnarray}
The connection homomorphism in degree $0$ is
\begin{align*}
  \delta^0\colon H^0(G(F_v),\Z)\lra H^1(G(F_v),\St_v^M).
\end{align*}
Since $G(F_v)$ acts trivially on $\Z$ we have that $H^0(G(F_v),\Z)\simeq \Z$. We can thus define $c_{\ell,v}\in H^1(G(F)^+,\St^M_v)$ to be the image of $\delta^0(1)$ under the restriction map $H^1(G(F_v),\St^M_v)\lra H^1(G(F)^+,\St^M_v)$.

% \begin{lemma}
% Let $S$ be a finite set of non-archimedean places of $F$, and let $\mathcal{L}=\prod_{v\in S}\St_v^\Z$. For any flat ring homomorphism $R\ra R'$ and any $R$-module $N$ there is an isomorphism
% \begin{align*}
%  H^q(G(F)^+, \cA^S(\cL\otimes_\Z R,N)\otimes_{R}R')\lra   H^q(G(F)^+, \cA^S(\cL\otimes_{\Z}R',N\otimes_{R}R')).
% \end{align*}
% \end{lemma}
% \begin{proof}
%  This is proved as in [Lehnart's thesis]\fixme{fix this}.
% \end{proof}

Now we consider the case where $\ell\colon F_v^\times \ra \Z$ is $\ell(x)=\ord_v(x)$, and denote by $c_{\text{ord},v}$ the cohomology class constructed above. %Let $S'$ be a finite set of non-archimedean places and $S\subseteq S'$ any subset. 
Let $V$ be a $G(F_{S\setminus\{v\}})$-module over $\Z$. Then the cup product  $\cup c_{\ord,v}$ provides a connection homomorphism
\begin{align*}
\delta_v^\ord\colon  H^q(G(F)^+,\cA^{S}(V\otimes\St_v^\Z,\Z))\lra H^{q+1} (G(F)^+,\cA^{S}(V,\Z)).
\end{align*}

\subsection{Non-archimedean evaluation morphisms}

Let $L$ be any non-trivial extension of $F_v$, and $\ell:L^\times\ra M$ a continuous homomorphism of topological abelian groups. Let us consider the $L$-rational points of the $v$-adic upper half plane $\cH_v=\cH_v(L)=\PP^1(L)\setminus\PP^1(F_v)$. Let $\Delta_v =\Div( \cH_v)$ and $\Delta_v^0 =\Div^0( \cH_v)$. We have a well defined $\GL_2(F_v)$-invariant morphism
\begin{align}
  \label{eq:pairing-coeffs}
 z\longmapsto (f_z,1)\colon\Delta_v\longrightarrow \cE(\ell);\qquad\mbox{where}\quad f_z\smtx a b c d=\ell(c z+d).
\end{align}
Such morphism induces a morphism of $G(F)$-modules
\begin{align*}
  \ev_v^\ell \colon \cA^S( X\otimes_{\Z}\cE(\ell), M)\lra \cA^{S}(X\otimes_{\Z} \Delta_v,M ),
\end{align*}
for any $G(F_{S\setminus\{v\}})$-module $X$ over $\Z$.
Similarly as before, the evaluation map induces a $G(F)$-module morphism 
\[
\ev_v^\ell \colon\cA^{S}(X\otimes_{\Z}\St_v^M,M)\ra \cA^{S}( X\otimes_{\Z}\Delta_v^0,M)
\]
making following diagram of $G(F)$-modules  commutative:
\begin{align}
  \label{eq:evaluation-nonarch}
\xymatrix{
0\ar[r]&\cA^S(X\otimes_{\Z}M) \ar[r]\ar[d]^{\mathrm{id}}&\cA^{S}(X\otimes_{\Z}\cE(\ell),M)\ar[r]\ar[d]^{\ev_v^\ell}& \cA^{S}(X\otimes_{\Z}\St_v^M,M) \ar[r]\ar@{-->}[d]^{\ev_v^\ell}&0\\
0\ar[r]& \cA^{S}(X\otimes_{\Z}M) \ar[r]&\cA^{S}( X\otimes_{\Z}\Delta_v,M)\ar[r]&\cA^{S}( X\otimes_{\Z}\Delta_v^0,M)\ar[r]&0.
}
\end{align}
Here, as before,  the bottom exact sequence is induced by the exact sequence defining the degree $0$ divisors $0\ra \Delta^0_v \ra \Delta_v\ra \Z\ra 0.$

\begin{remark}
  \label{rmk:adding-primes}
  Let $w\not\in S$ be a non-archimedean prime at which $B$ is split, and denote by $S'=S\cup\{w\}$. Let $U_w$ be the Hecke operator attached to the double coset $K_0(w)\smtx{\varpi}{0}{0}{1} K_0(w)$, where $\varpi$ is a local uniformizer for $w$ and $K_0(w)$ is the usual open compact subgroup of matrices which are upper-triangular modulo $\varpi$ at $w$. Let $M$ and $N$ be $\Z$-modules as above. We can consider the composition
  \begin{align*}
    H^r(G(F)^+,\cA^S(M,N))^{F_w^\times K_0(w),U_w}&\stackrel{\cong}{\longrightarrow} H^r(G(F)^+,\cA^{S'}(M\otimes \St^\Z,N))\\
    &\stackrel{c_{\text{ord}}}{\longrightarrow} H^{r+1}(G(F)^+,\cA^{S'}(M,N)),
  \end{align*}
  coming from the identification $\St^\Z=\Coind_{\overline{K_0(w)}}^{\PGL(F_w)} 1/(U_w-1)$ and Remark~\ref{AvsHom}. If $\Hom(M,N)$ has the additional structure of an $L_\pi$-module, the composition above induces an isomorphism
  \[
    H^r(G(F)^+,\cA^S(M,N))_\pi\cong H^{r+1}(G(F)^+,\cA^{S'}(M,N))_{\pi},
  \]
which is functorial in $M$ and $N$.
\end{remark}

\section{Automorphic periods and abelian varieties}\label{section: oda}%\fixme{(Marc) Compro la proposta del Xevi. Sempre podem fer servir el joc de paraules quan fem xerrades\ldots}

In this section we define lattices of periods attached to weight 2 automorphic representations. We conjecture that the tori (archimedean or non-archimedean) attached to these lattices are in fact algebraic, and that they give a concrete form of a conjecture formulated in~\cite{Taylor} by Taylor. At the end of the section we survey the known cases of the conjecture.

%Let $G$ be the multiplicative group of a quaternion algebra over $F$, and 
Let $\pi$ be an automorphic representation of $G(\A_F)$ of weight 2, and denote by $\Pi$ its Jacquet--Langlands lift to $\GL_2(\A_F)$. We continue to denote $S$ a finite set of split places containing $\SigmaB$, and assume furthermore that the local representation at every finite place in $S$ is Steinberg.% (when $\sigma$ is non-archimedean) or discrete series (when $\sigma$ is archimedean, and thus real).

Let $r=\#S$. Fix a place $v\in S$ such that $\pi_v$ is either Steinberg (if $v$ is non-archimedean) or discrete series (and thus $v$ real archimedean). % and set $\SigmaB^v=\SigmaB\setminus \{v\}$. We remark that it is possible that $v\not\in \SigmaB$, in which case $\SigmaB^v=\SigmaB$
Fix a character  $\lambda:G(F_{\SigmaB})/G(F_{\SigmaB})^+\rightarrow\{\pm 1\}$ such that if $v\in\SigmaB$ then  $\lambda(G(F_v))=1$. For a $G(F)$-module $M$ there is a natural action of $G(F)/G(F)^+$ on the cohomology groups $H^q(G(F)^+,M)$, and we write $  H^q(G(F)^+,M)^\lambda$ for the subspace on which $G(F)/G(F)^+$ acts through $\lambda$. We also denote by $\lambda^-$ the character obtained from $\lambda$ by sending $G(F_v)\setminus G(F_v)^+$ to $-1$ instead.

\begin{proposition}\label{prop: iso of the isotypical part}
Let $V_v=\St_v^\C$ be the Steinberg representation if $v$ is non-archimedean, or $V_v=D_v$ the discrete series representation if $v$ is archimedean. Let $\delta_v$ be the connection morphism $\delta_v^\ord$ if $v$ is non-archimedean, or $\delta_v^+$ if $v$ is real archimedean. Then $\delta_v$ induces an isomorphism
\[
\delta_v\colon H^{r-1}(G(F)^+,\cA^{S}(V_v,\C))^{\lambda}_\pi\stackrel{\cong}{\longrightarrow} H^r(G(F)^+,\cA^{S}(\C))^\lambda_\pi
\]
of one-dimensional $\C$-vector spaces.

\end{proposition}
\begin{proof}
If $v$ is non-archimedean, this result is given in \cite[Lemma 5.2 $(b)$]{Spiess} for $G=\GL_2$ over totally real fields $F$, but the proof can be extended to general $F$ and $G$. If $v$ is archimedean, the result follows from the Generalized Eichler-Shimura isomorphism \cite[Corollary 2.2]{Hida}, and the description of the Eichler-Shimura isomorphism given in \cite[\S 2.1]{molina}. 
\end{proof}

\begin{definition}
  Let $L_\pi$ be the field of definition of $\pi$, and let $R_\pi$ be the  ring of integers of $L_\pi$. Write
  \[
  X:=\Hom(R_\pi,\Z),
  \]
which is a free $\Z$-module of rank $[R_\pi\colon\Z]$.
\end{definition}

There is an isomorphism (see for example~\cite[Proposition 4.6]{Spiess}))
\[
  H^r(G(F)^+,\cA^S(\C))^\lambda \cong H^r(G(F)^+,\cA^S(R_\pi))^\lambda \otimes_{R_\pi} \C,
\]
so by Hecke theory there exists an element of $H^{r}(G(F)^+,\cA^{S}(R_\pi))^\lambda$ generating the subspace $H^{r}(G(F)^+,\cA^{S}(L_\pi))^\lambda(\pi)$. Composition with the canonical morphism $R_\pi\ra \Hom(X,\Z)$ provides an element $\tilde\phi\in H^{r}(G(F)^+,\cA^{S}(X,\Z))^\lambda_\pi$.

Note that there is an isomorphism between
\[
H^{r}(G(F)^+,\cA^{S}(R_\pi))^\lambda\otimes\Q\cong H^{r}(G(F)^+,\cA^{S}(X,\Z))^\lambda\otimes\Q,
\]
given by the natural isomorphism $L_\pi=R_\pi\otimes\Q\ra \Hom(X,\Z)\otimes\Q$.
\begin{proposition}
  Define  $M_v$ and $\cO_v$ as:
\begin{itemize}
\item $(M_v,\cO_v)=(\St_v^\Z,\Z)$ if $v$ is non-archimedean, and
\item $(M_v,\cO_v)=(D_v,\C)$ if $v$ is archimedean.
\end{itemize}

Then there exists a unique class
\[
  \phi\in H^{r-1}(G(F)^+,\cA^S(X\otimes M_v, \cO_v))^\lambda
\]
such that $\delta_v(\phi)=\tilde\phi$ and such that $\phi$ generates
\[
\left(  H^{r-1}(G(F)^+,\cA^S(X\otimes M_v, \cO_v))^\lambda\otimes_{R_\pi}\C\right)(\pi).
\]

\end{proposition}
\begin{proof}
  Proposition~\ref{prop: iso of the isotypical part} gives a class $\phi'\in H^{r-1}(G(F)^+,\cA^S(V_v,\C))^\lambda_\pi$ such that $\delta_v(\phi')=\tilde\phi$. The canonical isomorphism
  \[
    \Hom(X,\Z)\otimes_{R_\pi} \C\cong \C
  \]
  induces an isomorphism
  \[
H^{r-1}(G(F)^+,\cA^S(X\otimes M_v,\cO_v))\otimes_{R_\pi}\C\cong  H^{r-1}(G(F)^+,\cA^S(V_v,\C)).
\]
(In the non-archimedean case, we need to invoke~\cite[Proposition 4.9]{thesis-lennart}). The result follows from Hecke theory.
\end{proof}

\begin{remark}\label{rmk:Xint}
  The construction of the multiplicative integrals appearing in~\cite{Dar-int} allows for the definition of a morphism
  \begin{align*}
     \Hom(\St_v^\Z,\Z)\lra \Hom(\St_v^{L^\times},L^\times)
  \end{align*}
  that sends $\mu$ to the homomorphism $\displaystyle f\mapsto \lim_\mathcal{U} \prod_{U\in \cU}f(t_U)^{\mu(U)}$, where $\cU$ runs over coverings of $\mathbb{P}^1(F_v)$ by open compacts whose diameter tends to $0$ and $t_U$ is any sample point in $U$. It gives rise to a natural $G(F)$-invariant morphism, sometimes called multiplicative integral:
\begin{align*}
  \Xint\times \colon \cA^S(X\otimes_\Z\St_v^\Z,\Z)\lra \cA^S(X\otimes_\Z\St_v^{L^\times},L^\times),
\end{align*}
for any $G(F)$-module $X$ over $\Z$.
\end{remark}

Note that  there is a natural isomorphism $\cA^S(X\otimes_\Z\Delta_v^0,\C_v^\times)\cong \cA^S(\Delta_v^0,\Hom(X,\C_v^\times))$.
Therefore, we can consider the evaluation map
\begin{align*}
  \ev_v\colon H^{r-1}(G(F)^+,\cA^S(X\otimes M_v,\cO_v))^\lambda\lra H^{r-1}(G(F)^+,\cA^S(\Delta_v^0,\Hom(X,\C_v^\times))),
\end{align*}
where $\ev_v$ is $\exp(\ev_v^+ + \ev_v^-)$ if $v$ is archimedean (we write $\exp(z)=e^{2\pi i z}$), and otherwise $\ev_v$ is $\ev_v^{\ell}\circ\Xint\times$ as in Remark \ref{rmk:Xint} (with $L=\C_v$ and $\ell$ the identity on $L^\times$).

% If $v$ is arquimedean let $q=\exp(\tau)$ for $\tau\in \C$ any element such that $\tau^{-1}\delta_v^-(\phi)$ belongs to $H^{n+1}(G(F)^+,\cA_f^S(\Z))$.

The image of $\phi$ under  $\ev_v$ is a class
\[
    \psi\in H^{r-1}(G(F)^+,\cA^S(X\otimes_{\Z} \Delta_v^0, \C_v^\times)).
  \]

% Define $\delta_v$ to be the connection homomorphism
% \begin{align*}
% \delta^v_{\ord}\colon   H^{n}(G(F)^+, \cA_f^S(X\otimes_\Z \St_v^\Z,\Z))^\lambda\lra  H^{n+1}(G(F)^+,\cA_f^{S}(X,\Z))^\lambda
% \end{align*}
% if $v$ is non-archimedean, and to be
% \begin{align*}
% \delta_v^{\lambda(v)}\colon   H^{n}(G(F)^+, \cA^{\Sigma\cup S}(X\otimes_\Z D_v,\C))^\lambda\lra  H^{n+1}(G(F)^+,\cA_f^{S}(X,\C))^\lambda
% \end{align*}
% in case $v$ is archimedean. To lighten the notation, we put
% \begin{align*}
%   A_v = \begin{cases}\cA_f^S(X\otimes_\Z\St_v^\Z,\Z),\  \text{ if $v$ is non-archimedean}\\ \cA^{\Sigma\cup S}(X\otimes_\Z D_v,\C), \ \text{ if $v$ is archimedean.
% } \end{cases}
% \end{align*}
% By Proposition \ref{prop: iso of the isotypical part} there exists a unique preimage $\phi\in H^{n}(G(F)^+, A_v)^\lambda$ of $\tilde \phi $ that generates $\rho$ (or $\rho\otimes_\Q\C$ in the arquimedean case).

% If $v$ is non-archimedean, we have the following result.

Consider the connection morphism
\[
\partial:H^{r-1}(G(F)^+,\cA^S(\Delta_v^0,\Hom(X,\C_v^\times)))\longrightarrow H^{r}(G(F)^+,\cA^S(\Hom(X,\C_v^\times))),
\]
arising from the short exact sequence
\begin{align}
  \label{eq:ses-divisors}
  0\lra \Delta_v^0\lra\Delta_v\stackrel{\deg}{\lra} \Z\lra 0.
\end{align}
\begin{proposition}\label{prop:latticeTrivial}
  There exists a $R_\pi$-invariant lattice $Y\subset\Hom(X,\C_v^\times)$ of rank $[R_\pi:\Z]$ such that $\partial(\psi)$ belongs to the image of the map
  \[
    H^r(G(F)^+,\cA^S(Y))\otimes\Q \lra H^r(G(F)^+,\cA^S(\Hom(X,\C_v^\times)))\otimes\Q
  \]
  induced by the inclusion.
  % \begin{align*}
  %  \partial(\psi)\in H^{r}(G(F)^+,\cA^S(Y))\otimes \Q.
  % \end{align*}

\end{proposition}
\begin{proof}
  \textit{Non-archimedean case: }
  Suppose that $v$ is non-archimedean. As proven in~\cite[Equation (4.3)]{lennart-felix}, there exists a ``automorphic period'' $q=q_{\pi,v}\in F_v^\times\otimes_{\Z} R_\pi$, which is unique up to torsion, such that
  \[
    n (\phi \cup c_{\ord,v}) \otimes q = n \cdot \ord_v(q) \left(\Xint\times \phi\right)\cup c_{\operatorname{id},v}
  \]
  for sufficiently large $n\in\Z_{\geq 1}$. Here $\operatorname{id}\colon F_v^\times\to F_v^\times$ is the identity map. We view $q$ as an element of $\Hom(X,F_v^\times)$ using the canonical injection $F_v^\times\otimes_{\Z} R_\pi\hookrightarrow \Hom(X,F_v^\times)$. Finally, define $Y\subset \Hom(X,F_v^\times)$ as the $R_\pi$-module generated by $q_{\pi,v}$. In~\cite[Remark 4.7(iii)]{lennart-felix} the authors show that $\ord(q_{\pi,v})\neq 0$ (as an element of $R_\pi$), and thus that $Y$ is indeed a lattice.

  The long exact sequences in cohomology obtained from the diagram in~\eqref{eq:evaluation-nonarch} yield a commutative diagram
  \[
    \xymatrix{
      H^{r-1}(G(F)^+,\cA^S(X\otimes \St_v^\Z,\Z))\ar[r]^-{\Xint\times(\cdot)\cup c_{\operatorname{id},v}}\ar[d]^{\ev_v}&H^r(G(F)^+,\cA^S(X,F_v^\times))\ar[d]\\
      H^{r-1}(G(F)^+,\cA^S(\Delta^0_v,\Hom(X,\C_v^\times)))\ar[r]^{\partial}&H^r(G(F)^+,\cA^S(\Hom(X,\C_v^\times))).
      }
  \]
  This gives $\partial(\psi) = \Xint\times \phi \cup c_{\operatorname{id},v}$. Therefore $\partial(\psi)$ belongs to $H^r(G(F)^+,\cA^S(Y))\otimes\Q$, as sought.

  \textit{Archimedean case: }
  Suppose now that $v$ is real archimedean. By~\eqref{eq:evaluation-archimedean}, we have
  \[
    (\partial\circ\ev_v^{+})(\phi)=\tilde\phi\in H^r(G(F)^+, \cA^S(X,\Z)).
  \]
  Since $\exp(\Z)=1$, we get
  \[
    \partial(\psi)=(\partial\circ\ev_v)(\phi)=\exp\left(\partial(\ev_v^+(\phi))+\partial(\ev_v^-(\phi))\right)=\exp(\partial(\ev_v^-(\phi))=\exp(\delta_v^-(\phi)).
    \]
    Note that $\delta_v^-(\phi)$ belongs to $H^r(G(F)^+,\cA^S(X,\C))(\pi)^{\lambda^-}$. Since
    \[
      H^r(G(F)^+,\cA^S(X,\C))=H^r(G(F)^+,\cA^S(R_\pi))\otimes_{\Z}\C,
    \]
    we can view $\delta_v^-(\phi)$ as belonging to the latter cohomology group. By choosing a generator $\varphi$ of the $\pi$-isotypical component of  $H^r(G(F)^+,\cA^S(R_\pi))(\pi)^{\lambda^-}$, we may write
    \[
      \delta_v^-(\phi) = \tau\cdot\varphi,\quad\tau\in R_\pi\otimes_{\Z}\C\cong \C^{[L_\pi\colon\Z]}.
      \]
    By identifying $\exp(\tau)\in \C^\times\otimes_{\Z} R_\pi$ with an element of $\Hom(X,\C^\times)$ using the canonical injection $R_\pi\otimes_{\Z}\C\hookrightarrow\Hom(X,\C)$ followed by $\exp$, we define $Y$ as the $R_\pi$-module that it generates. It is then clear that $\partial(\psi)$ belongs to $H^r(G(F)^+,\cA^S(Y))$.

    It remains to show that $Y$ is indeed a lattice. Note that $Y$ is $\exp(R_\pi+ \tau R_\pi)$, so it is enough to show that $R_\pi + \tau R_\pi\subseteq \Hom(X,\C)$ is a lattice. We will prove a slightly stronger fact, namely that $\tau$ belongs to $R_\pi\otimes i\R$. By~\cite[Lemma 2.2]{molina}, we have
    \begin{equation}
      \label{eq:molina}
      \overline{\partial^\pm(\phi)} = \pm \partial^{\pm}(\bar\phi).
    \end{equation}

    Therefore we deduce
    \[
      \delta_v^+\phi=\tilde\phi = \overline{\tilde\phi}=\overline{\delta_v^+\phi}=\delta_v^+\overline{\phi},
    \]
    and thus $\phi=\overline{\phi}$. In the previous display, the second equality holds because conjugation in $H^r(G(F)^+,\cA^S(R_\pi))\hookrightarrow H^r(G(F)^+,\cA^S(R_\pi)\otimes_{R_\pi}\C)=H^r(G(F)^+,\cA^S(\C))$ happens only on the $\C$-term.

    Another application of Equation~\eqref{eq:molina} yields
    \[
      \bar\tau\varphi = \bar\tau\overline{\varphi}=\overline{\delta_v^-\phi}=-\delta_v^-\overline{\phi}=-\delta_v^-\phi=-\tau\varphi,
    \]
    and so $\bar \tau=-\tau$, as wanted.

%   [Sketch of the proof in the archimedean case\fixme{Not finished, not even sure this is the way to go}] We work with $\C$ rather than $\C^\times$. We have that $\partial (\ev_v(\phi))=\partial(\ev^+\phi) + \partial(\ev_-\phi)=\partial^+\phi + \partial^-\phi$. Now, by construction we have that $\partial^+\phi\in H^{n+1}(G^+\cA_f^S(R_\pi))^\lambda$, and $R_\pi$ can be identified with a discrete subgroup of rank $r$ of $\Hom(X,\C)$, because $\Hom(X,\C)\simeq R_\pi\otimes \C$.

% On the other hand, $\partial^-\phi$ belongs to $H^{n+1}(G^+\cA_f^S(R_\pi))^{\lambda^-}$, where $\lambda^-(v)=-1$ and $\lambda^-(w)=\lambda(w)$ for $w\neq v$. Now this space has dimension $1$ as an $L_\phi\otimes\C$-vector space, generated by $\tilde\phi_-$. Therefore, there exists $\tau\in R_\phi\otimes\C$ such that $\partial^-\phi = \tau \tilde\phi_-$. Now we define $Y=R_\pi + \tau R_\pi\subset R_\pi\otimes \C$. This should be the guy.
\end{proof}

Write $T_Y=\Q\otimes_{\Z}(\Hom(X,\C_v^\times)/Y)$, and consider the exact sequence
\[
  H^{r-1}\left(G(F)^+,\cA^S(\Delta_v,T_Y)\right)\lra H^{r-1}\left(G(F)^+,\cA^S(\Delta_v^0,T_Y)\right)\lra H^{r}\left(G(F)^+,\cA^S(T_Y)\right)
\]
arising again from~\eqref{eq:ses-divisors}. The previous proposition shows that there exists an element
\[
  \Psi\in   H^{r-1}\left(G(F)^+,\cA^S(\Delta_v,T_Y)\right)^\lambda
  \]
whose image in $H^{r-1}\left(G(F)^+,\cA^S(\Delta_v^0,T_Y)\right)$ coincides with the image of $\psi$ under the map induced from the quotient $\Hom(X,\C_v^\times)\ra T_Y$, as illustrated in the following diagram:
\[
    \xymatrix{
  H^{r-1}\left(G(F)^+,\cA^S(\Delta_v,\Hom(X,\C_v^\times))\right)\ar[r]\ar[d] &      H^{r-1}\left(G(F)^+,\cA^S(\Delta_v,T_Y)\right)\ni\Psi\ar[d] \\
\psi\in H^{r-1}\left(G(F)^+,\cA^S(\Delta_v^0,\Hom(X,\C_v^\times))\right)\ar[r]\ar[d]^{\partial}&      H^{r-1}\left(G(F)^+,\cA^S(\Delta_v^0,T_Y)\right)\ar[d]^{\partial} \\
 H^{r}\left(G(F)^+,\cA^S(\Hom(X,\C_v^\times))\right)\ar[r] &H^{r}\left(G(F)^+,\cA^S(T_Y)\right).
      }
    \]

Note that the previous construction can be carried out on the $\pi$-components, since all the maps in sight respect the $G(\A_F^S)$-action. Therefore we may and do assume that $\Psi$ generates $H^{r-1}\left(G(F)^+,\cA^S(\Delta_v,T_Y)\right)(\pi)$. The following Lemma ensures unicity of such a $\Psi$.

\begin{lemma}
 $H^{r-1}(G(F)^+,\cA_f^S(T_Y))_\pi=0.$
\end{lemma}
\begin{proof}
Note that $T_Y$ is a $L_\pi$-vector space, thus the result follows from the fact that $H^{r-1}(G(F)^+,\cA_f^S(L_\pi))_\pi$ is zero since, by Proposition \ref{prop: iso of the isotypical part}, $r$ is the lowest cohomology group where $\pi$ appears.
\end{proof}

 Recall the variety $A_\Pi/F$ alluded to in the introduction, which is conjecturally attached to $\Pi$ (cf. \cite{Taylor}). In general $ A_\Pi$ is expected to have dimension $[L_\Pi\colon \Q]$ and to satisfy that $\mathrm{End}(A_\Pi\otimes \bar \Q)\otimes \Q\simeq L_\Pi$. However, when $F$ is totally imaginary then it can also be the case that  $\dim A_\Pi=2[L_\Pi\colon\Q]$ and $\mathrm{End}(A_\Pi\otimes \bar \Q)\otimes \Q \simeq  D$ is some quaternion division algebra with center $L_\Pi$. If $L_\Pi$ is totally real, the arguments of \cite[\S 2]{GM} rule out this case under our running assumption that there exists at least one place $v$ at which $\Pi$ is Steinberg or discrete series. %\fixme{(Xevi): esteu d'acord amb aquesta frase?}. 
 Since in the present article we are allowing $\omega_\Pi$ to be non-trivial, $L_\Pi$ can also be a CM field in our setting, but the arguments of \cite[\S 2]{GM} can be easily adapted to treat also this case.

To begin with, if $v$ is archimedean then by assumption $v$ is real and therefore $A_\Pi$ has dimension $[L_\Pi\colon \Q]$ in this case. Suppose now that $v$ is non-archimedean, and that $\dim A_\Pi=2[L_\Pi\colon \Q]$ and $\mathrm{End}(A_\Pi\otimes \bar \Q)\otimes \Q \simeq  D$ is a quaternion algebra as above. Then we can find a field $K$ of degree $\dim A_\Pi$ such that $K\hookrightarrow \mathrm{End}(A_\Pi\otimes \bar \Q)\otimes \Q $. We are assuming that $v$ divides exactly the level of the newform associated to $\Pi$, so that the conductor of $A_\Pi$ is divisible exactly by $v^{\dim A_\Pi}$. Now a similar argument to that in \cite[Proposition 2.4]{GM} shows that $A_\Pi$ has purely multiplicative reduction at $v$. Indeed, \cite[Proposition 2.4]{GM} proves this assertion if $K$ is totally real but this assumption is only used to deduce that the determinant of (one of the factors of) the $\ell$-adic representation of $V_\Pi$ is the cyclotomic character; in our situation, and according to the properties that one expects for the conjectural Galois representation attached to $\Pi$, the determinant equals the product of the cyclotomic character with the central character $\omega_\Pi$, which by assumption is also unramified at $v$ and therefore the rest of the argument applies without change. Finally, note that if $L_\Pi$ is a CM field then $D$ is totally indefinite, and one can see as in the proof of \cite[Proposition 2.2]{GM} that $A$ has potentially good reduction at $v$, which gives then a contradiction. 
\begin{remark}\label{remark:A_Pi}
To sum up the previous discussion, we record that under our running assumptions $A_\Pi$ is always of dimension $[L_\Pi\colon \Q]$; moreover, if $v$ is non-archimedean then $A_\Pi$ has purely multiplicative reduction at $v$ and therefore it admits a $v$-adic uniformization.
\end{remark}
\begin{conjecture}\label{conj: oda}
  There exists a $R_\pi$-equivariant isogeny between $\Hom(X,\C_v^\times)/Y$ and $A_\Pi/\C_v$.
\end{conjecture}

This conjecture is known  to hold for totally real base fields $F$, in certain cases:
\begin{itemize}
  \item When $\Sigma_B=\{v\}$ Conjecture~\ref{conj: oda} is known: thanks to Remark~\ref{rmk:adding-primes}, it is enough to consider in this case the situation $S=\Sigma_B$, and then this is the statement of the Eichler--Shimura isomorphism in the setting of Shimura curves. See for instance~\cite[\S 4.1]{Santi}.

  \item When $v$ is archimedean, $B=M_2(F)$ for a real quadratic $F$ and $\pi$ is the base-change of a modular form of weight $2$ with quadratic Hecke field, then Conjecture~\ref{conj: oda} is a result of Oda, see~\cite[Main Theorem B]{oda-book}.

  \item Under certain hypothesis (``SNV''), and when $v=\mathfrak{p}$ is a finite prime of prime norm $p$ satisfying that $\pi$ is ordinary at all primes dividing $p$, and that it admits a Jacquet--Langlands lift to a totally-definite quaternion algebra, then the conjecture follows from~\cite[Theorem B]{lennart-L-invariants} and~\cite[Theorem 4.9]{lennart-felix}%\fixme{(Xevi): no ho entenc! a mi em sembla que\cite{lennart-L-invariants} i \cite{lennart-felix} son el mateix article! He anat a la pagina d'en Lennart http://www.esaga.uni-due.de/lennart.gehrmann/papers/ i no em sembla pas que en Lennart tingui un article ell sol a Transactions of the AMS, es conjunt amb en Felix Bergunde. Si que es cert que a la pagina d'en Lennart canvia una mica el titol (treu la part de Leading terms), pero si segueixes el link vas a l'article amb en Felix. Com es que tenim posades dues referencies diferents???} 
  , which strengthen the results of \cite{Dar-int} and \cite{Spiess} in this direction.
  \end{itemize}

For the rest of the paper, we assume that Conjecture \ref{conj: oda} holds. Therefore we will regard $\Psi$ as an element of
\begin{align}
  \label{eq:definition-of-Psi}
  \Psi=\Psi^\lambda&\in   H^{r-1}\left(G(F)^+,\cA^S(\Delta_v,A_\Pi(\C_v)\otimes\Q)\right)^\lambda_\pi.
\end{align}

\section{Darmon points}\label{D-P}

The goal of this section is to introduce a family of points attached to optimal embeddings of quadratic extensions $K$ of our base field $F$. These generalize all known constructions of the so-called ``Darmon points'', as treated in~\cite{GMS}. In addition to their algebricity and a Shimura reciprocity law, we have not been able to resist the temptation to conjecture a Gross--Zagier-type formula as well. We hope that further experiments will provide evidence for the validity of such formula in our setting.

\subsection{Automorphic cohomology classes}
Let $\Pi$ be a cuspidal automorphic representation of $\GL_2(F)$ of parallel weight $2$ and let $K$ be a quadratic extension of $F$. Define $\Sigma_{\ns}$ as
\[
\Sigma_{\ns}:=\{w\mid \infty;\;w\mbox{ does not split in }K\},
\]
and choose a set of finite places
\[
S_{\ns}\subseteq\{w\nmid \infty;\;w\mbox{ does not split in }K\;\mbox{and }\Pi_w\mbox{ is Steinberg}\}.
\]
Let us assume that the cardinality of $\Sigma_{\ns}\cup S_{\ns}$ is odd.
Fix $v\in \Sigma_{\ns}\cup S_{\ns}$ and let $G$ be the multiplicative group of a quaternion algebra $B$ over $F$ with ramification set $(\Sigma_{\ns}\cup S_{\ns})\setminus\{v\}$. Fix a splitting $\iota_v\colon B\hookrightarrow M_2(F_v)$. Let $\pi$ be the Jacquet-Langlands transfer of $\Pi$ to $G$, which exists by our choice of ramified places in $B$. Let
\[
\Sigma_{s}:=\{w\mid \infty;\;w\mbox{ splits in }K\},
\]
and choose a set of finite places
\[
S_{s}\subseteq\{w\nmid \infty;\;w\mbox{ splits in }K\;\mbox{and }\pi_w\mbox{ is Steinberg}\}.
\]

Fix also an embedding $\psi\colon K\hookrightarrow B$, which will allow us to  identify $K$ with its image by $\psi$ in $B$. We  will write $K^+\subset K^\times$ for the intersection $K^+=K^\times \cap G(F)^+$. Let $S=\Sigma_s\cup S_s\cup \{v\}$, and choose a character $\lambda\colon K^\times/K^+\ra\{\pm 1\}$ which is equivalent, via $\psi$, to a character of $G(F_{\Sigma_B})/G(F_{\Sigma_B})^+$. In~\eqref{eq:definition-of-Psi}  we have attached to such a character and to the representation $\pi$ an element $\Psi=\Psi^\lambda$ in
\[
  H^{r-1}\left(G(F)^+,\cA^S(\Delta_v,A_\Pi(\C_v)\otimes\Q)\right)^\lambda_\pi.
  \]

% Let $\Sigma=\Sigma_s\cup\Sigma$ be the set of archimedean places and let $S:=S_s\cup S_{\ns}$. Define $\Sigma_s^\ast$ (resp. $S_s^*$) be the set of archimedean (resp. non-archimedean) places defined by the relation
% \begin{align*}
% \Sigma_s^*\cup S_s^* = \Sigma_s\cup S_s\cup\{v\}.
% \end{align*}

% By the choices we have done, we have that
% \[
% H^0(G(F),\cA^{\Sigma\cup S}(V_{\Sigma\cup S},\C))_\pi\simeq\C,
% \]
% where
% \[
% V_{\Sigma\cup S}=\bigotimes_{\sigma\in\Sigma_{\ns}\setminus\{v\}}\C\otimes \bigotimes_{\sigma\in\Sigma_{s}^*}D_\sigma\otimes \bigotimes_{u\in S_{\ns}\setminus \{ v\}}\C\otimes \bigotimes_{u\in S_{s}^*}{\rm St}_u.
% \]
% Fix a choice of signs $\lambda$ as in \eqref{eq:lambda} with $\lambda(v)=1$ if $v$ is archimedean, and $\ell_w=\ord_w$ for $w\in S_s^*$, we apply the connection morphisms
% \begin{align*}
% H^0(G(F),\cA^{\Sigma\cup S}(V_{\Sigma\cup S},\C))_\pi&\stackrel{\delta^\lambda}{\longrightarrow}H^r(G(F)^+,\cA_f^{S}(\bigotimes_{v\in S_{s}^*}{\rm St}_v,\C))_\pi^\lambda\\
% &\stackrel{\delta_\ord^{S_s^*}}{\longrightarrow}H^{n+1}(G(F)^+,\cA_f^{S}(\C))_\pi^\lambda,
% \end{align*}
% where $r=\#\Sigma_s^*$ and $n=\# S_s+\#\Sigma_s$.
% \begin{proposition}
% The above arrows are isomorphisms of one dimensional $\C$-vector spaces.
% \end{proposition}
% \begin{proof}
% Article Spiess + Article Harder.
% \end{proof}

% Fix an embedding $\psi:K^\times\hookrightarrow G(F)$, that exist by the assumptions we have considered.

Let $ z\in \Delta_v$ be the unique element fixed by $K^\times$ such that $\iota_v(x)( z,1)^t=x( z,1)^t$ for all $x\in K$.
\newcommand{\CoInd}{\operatorname{CoInd}}
\begin{lemma}
For an abelian group $M$ we have a natural $G(F)^+$-equivariant isomorphism
\[
\rho\colon\cA^S(\Z[G(F)^+/K^+],M)\longrightarrow \CoInd_{K^+}^{G(F)^+}\cA^{S}(M).
\]
\end{lemma}
\begin{proof}
  By definition, the left-hand side is the module whose elements are
  \[
\cA^S(\Z[G(F)^+/K^+],M) = \{f\colon \Z[G(F)^+/K^+]\times G(\A_F^S) \ra M\colon f\text{ is continuous}\}
\]
and the right-hand side consists of elements
  \[
 \CoInd_{K^+}^{G(F)^+}\cA^S(M) = \{F\colon G(F)^+\times G(\A_F^S) \ra M\colon F(k\gamma,g)=F(\gamma,k^{-1}g)\colon F\text{ is continuous}\}.
\]
The isomorphism $\rho$ is then given by
\[
  \rho(f)(\gamma,g)=f(\gamma^{-1},\gamma^{-1}g),\quad \rho^{-1}(F)(\gamma,g)=F(\gamma^{-1},\gamma^{-1}g).
  \]
One easily checks that this is indeed $G(F)^+$-equivariant.
\end{proof}
Notice that $\Z[G(F)^+/K^+]\simeq \Z[G(F)^+ z]\subset\Delta_v$.
Hence if we restrict $\Psi$ to $\Z[G(F)^+ z]$ we obtain
\begin{eqnarray*}
\Psi_K&\in& H^{r-1}(G(F)^+,\cA^S(\Z[G(F)^+/K^+],A_\Pi(\C_v)\otimes\Q))^\lambda_\pi\\
&\simeq& H^{r-1}(G(F)^+,\CoInd_{K^+}^{G(F)^+}\cA^S(A_\Pi(\C_v)\otimes\Q))^\lambda_\pi\\
&\simeq& H^{r-1}(K^+,\cA^S(A_\Pi(\C_v)\otimes\Q))^\lambda_\pi,
\end{eqnarray*}
where the last equality has been obtained by Shapiro's Lemma. % Observe that the choice of signs $\lambda$ provides a character $\lambda\colon K^\times/K^+\ra \{\pm 1\}$. By construction, $\Phi_K$ in fact belongs to
% \begin{align*}
% H^{n}(K^+,\cA_f^S(A))^\lambda\simeq H^{n}(K^\times,\cA_f^S(A)(\lambda)).
% \end{align*}
%  We denote by $\phi_K\in H^{n}(K^\times,\cA_f^S(A)(\lambda))_\pi$
%  the homomorphism that sends a generator of $\rho\mid_{G(\A_F^{S,\infty})}$ to $\Phi_K$.

\subsection{Homology classes arising from class field theory} 
Write $T$ the algebraic group associated to $K^\times/F^\times$, namely $T(R)=(K\otimes_F R)^\times/(R)^\times$, for any $F$-algebra $R$. Write also $T(F)^+=K^+/F^\times$.

Let $\cO=\cO_{S}$ be a maximal compact subgroup of $T(\A_F^{S})$, and let $\Gamma$ be the intersection $T(F)^+\cap \cO$. The class group $(T(F)^+\cO)\backslash T(\A_F^{S})$ is finite, and we let $\{g_i\}_{i=1,\ldots, h}$ be a system of coset representatives. This choice of representatives gives rise to a fundamental domain $\cF$ for $\cO\backslash T(\A_F^{S})$ under the action of $\Gamma\backslash T(F)^+$, satisfying
\[
  \cF = \coprod_{i=1}^h g_i \cO\subset T(\A_F^{S}).
  \]

The goal of this $\S$ is to define a natural homology class
\[
  \xi\in H_{r-1}(T(F)^+, C_c^0(T(\A_F^{S}),\Z)).
\]
By Shapiro, this group is isomorphic to
\[
  H_{r-1}(\Gamma, C(\cF,\Z)),
\]
so we will instead construct a class in this latter group.

Consider first the group $H_{r-1}(\Gamma,\Z)$. The group $\Gamma$ is the group of relative $S$-units of $K$ and, since all places of $S$ except $v$ are split in $K$, the group $\Gamma$ has rank $\#S -1 = r-1$. If $r=1$, we choose $\zeta\in H_{0}(\Gamma,\Z)=\Z$ to be a generator (there are only two of them). If $r>1$, we choose $\zeta\in H_{r-1}(\Gamma,\Z)$ to be a preimage of a generator of $H_{r-1}(\Gamma/\text{torsion},\Z)$.

On the other hand, the characteristic function $1_{\cF}$ on $\cF$ gives rise to a generator of
\[
  H^0(\Gamma,C(\cF,\Z)).
\]
The cap product $\zeta\cap 1_{\cF}$ is the sought class in $H_{r-1}(\Gamma, C(\cF,\Z))$. See \cite[Section 1.1]{lennart-felix} for a similar construction.

Let $K_{S}^{\text{ab}}$ be the maximal %anticyclotomic 
abelian extension of $K$ which is unramified and totally split at the primes of $S$, and let $\cG =\Gal(K_{S}^{ab}/K)$. By Class Field Theory, the Artin map factors through
\[
  \A_K^\times/%\A_F^\times 
  K^\times\longrightarrow \left(K^\times / K^+\right) \times \left((\A_K^{S})^\times/ K^+\right)\stackrel{\rho_K}{\longrightarrow} \cG.
\]

 Hence by means of $\rho_K$, we can identify
\begin{align}\label{eq:functions with compact support}
C(\cG,\bar\Q)\simeq \bigoplus_{\beta\colon K^\times/K^+\to\{\pm 1\}} H^0(K^+, C((\A_K^S)^\times/U,\bar\Q)),
\end{align}
where $U\subset (\A_K^S)^\times$ is $(\A_K^S)^\times\cap K_\infty^\times$.
\subsection{Construction of points}
Note that the product provides a well defined $K^+$-equivariant pairing
\begin{equation}\label{cup-C-Cc}
  C_c(T(\A_F^S),\Z)\times C((\A_K^S)^\times,\bar\Q)\lra C_c^F((\A_K^S)^\times,\bar\Q),
\end{equation}
where $C_c^F((\A_K^S)^\times,\bar\Q)$ is the subset of functions $f\in C((\A_K^S)^\times,\bar\Q)$ with compact support modulo $(\A_F^S)^\times$, namely ${\rm Supp}(f)(\A_F^S)^\times\slash(\A_F^S)^\times$ is compact in $T(\A_F^S)$.

There is a $T(F)$-invariant pairing
\begin{align}\label{eq: pairing def}
\langle\ ,\ \rangle\colon  C_c^F((\A_K^S)^\times,\bar\Q)\times \cA^S(A_\Pi(\C_v)\otimes\Q)\lra  A_\Pi(\C_v)\otimes_{L_\pi} \bar\Q,
\end{align}
given by
\begin{align*}
  \langle f, \psi \rangle =\int_{(\A_K^{S})^\times}f(t)\psi(t) dt^\times
\end{align*}
where $dt^\times$ is the Haar measure normalized to give volume $1$ to the maximal compact subgroup of $(\A_K^S)^\times$. Observe that since $f$ is locally constant and compactly supported modulo $(\A_F^S)^\times$, the integral is actually a finite sum. This induces, by cap product, a pairing
\begin{align*}
   H_{r-1}(K^+,C_c^F((\A_K^S)^\times,\bar\Q))\times H^{r-1}(K^+,\cA^S(A_\Pi(\C_v)\otimes\Q))\stackrel{\langle\langle\cdot,\cdot\rangle\rangle}{\lra}  A_\Pi(\C_v)\otimes_{L_\pi} \bar\Q.
\end{align*}
\begin{definition}
  Given $\chi\in C(\cG,\bar\Q)$, denote by $\chi_\lambda\in H^0(K^+,C(T(\A^S),\bar\Q)$ its $\lambda$-component  as in Equation~\eqref{eq:functions with compact support}. For $f\in \rho|_{G(\A^S)}$, we define the Darmon point $P_\lambda(\chi,f)\in A_\Pi (\C_v)\otimes_{L_\pi} \bar\Q$ as
  \begin{align*}
    P_\lambda(\chi,f)= \langle\langle\xi \cap \chi_\lambda, \Psi_K(f)\rangle\rangle,
  \end{align*}
  where the cap product arises from the pairing in Equation~\eqref{cup-C-Cc}. %\fixme{(Xevi): aquesta referencia esta be???}
\end{definition}
\begin{conjecture}\label{conjecture_princ}
  \begin{enumerate}
  \item (Rationality) $P_\lambda(\chi,f)$ belongs to $A_\Pi(K_S^{\text{ac}})\otimes_{L_\pi} \bar\Q$.
    \item (Reciprocity law) For any $\sigma\in \cG$ the Galois action on $P_\lambda(\chi,f)$ is given by
  \begin{align*}
    P_\lambda(\chi,f)^\sigma = P_\lambda(\chi^\sigma,f),
  \end{align*}
where %\fixme{Atencio: mirar com normalitzem l'Artin map perque aixo sigui cert (i no sigui $\sigma$ en lloc the $\sigma^{-1}$)} 
$\chi^\sigma (\gamma)= \chi(\sigma^{-1}\gamma)$.
\item (Gross--Zagier-type formula) For any finite character $\chi\in C(\cG, \bar\Q)$ such that $\chi\circ{\rho_K}_{|_{K^\times/K^+}}=\lambda$, one has %\mid_{T(F_\infty)/U_{T_\infty}}=\lambda$, one has
  \begin{align*}
    \langle P_\lambda(\chi,f),P_\lambda(\chi,g) \rangle_{NT}= c \cdot L'(1/2,\pi,\chi)\cdot \prod_{u\not \in S}\alpha_{\pi_u,\chi_u}(f_u\otimes g_u),
  \end{align*}
where  $\langle\ ,\ \rangle_{NT}$ is the canonical Neron--Tate height pairing on $A_\Pi$, the quantity $c$ is a positive constant that depends on the various choices made in the construction, and  $\alpha_{\pi_u,\chi_u}(f_u\otimes g_u)$ is as defined in~\cite{Zhang}.
  \end{enumerate}
\end{conjecture}

\section{Appendix: Comparison with classical Heegner points}\label{appendix}

We assume that $F$ is a totally real number field and $K/F$ is a totally imaginary quadratic extension. Thus, with the notation of \S \ref{D-P}, we have that $\Sigma_{\rm ns}=\{w\mid \infty\}$. We choose a set of finite places $S_{\rm ns}$ as in \S \ref{D-P} such that the cardinality of $\Sigma_{\rm ns}\cup S_{\rm ns}$ is odd. We will proceed to compare the construction given above to the classical constructions of Heegner points.

\subsection{Archimedean uniformization of Heegner points}
 We choose $v\in \Sigma_{\rm ns}$ and let $G$ be the multiplicative group of the quaternion algebra with ramification set $(\Sigma_{\rm ns}\cup S_{\rm ns})\setminus\{v\}$. Let $\pi$ be the Jacquet-Langlands lift to $G$ of the parallel weight 2 cuspidal automorphic representation $\Pi$. Note that, in this case, there is only one possible choice of signs $\lambda$, and $\pi\mid_{G(\A_F^\infty)}$ is generated by $\phi\in H^0(G(F),\cA^{v}(X\otimes D_v,\C))^U$ where, as above, $X=\Hom(R_\pi,\Z)$ and $U\subset G(\A_F^\infty)$ is an open compact subgroup. Since $X\otimes\C\simeq \Hom(L_\pi,\C)$, we have an isomorphism
\[
H^0(G(F),\cA^{v}(X\otimes D_v,\C))\longrightarrow\prod_{\sigma:L_\pi\hookrightarrow\C}H^0(G(F),\cA^{v}(D_v,\C));\qquad \phi\longmapsto (\phi^\sigma)_\sigma, 
\] 
where $\phi^\sigma(f)=\phi(\sigma\otimes f)$, for all $f\in D_v$. Each $\phi^\sigma$ provides an holomorphic differential form $\omega^\sigma$ of a Shimura curve $X_U/F$, whose set of $\C$-valued points is in correspondence with the double coset space
\begin{equation}\label{Cparam}	
X_U(\C)=G(F)^+\backslash \cH\times G(\A_F^\infty)/U=\bigsqcup_{\gamma\in G(F)^+\backslash G(\A^\infty)/U}\Gamma_\gamma\backslash\cH,
\end{equation}
where $\Gamma_\gamma=\gamma U\gamma^{-1}\cap G(F)^+$. More explicitly, $\omega^\sigma(\gamma,z)=(2\pi i f_2)^{-1}\phi^\sigma(f_2)(g_v,\gamma)dz$, where $g_v\in G(F_v)$ satisfies $g_v i=z\in\cH$, and $f_2\in D_v$ is the holomorphic vector $f_2\mbox{\tiny$\left(\begin{array}{cc}a&b\\c&d\end{array}\right)$}=(ad-cb)(ci+d)^{-2}$.
By \cite[Theorem 2.4]{molina}, the connection homomorphisms
\[
\xymatrix{
H^0(G(F),\cA^{v}(X\otimes D_v,\C))\ar[d]^{\simeq}\ar[r]^{\delta_v^\pm}&H^1(G(F)^+,\cA^{v}(X,\C))\ar[d]^{\simeq}\\
\prod_{\sigma:L_\pi\hookrightarrow\C}H^0(G(F),\cA^{v}(D_v,\C))\ar[r]^{(\delta_\sigma^\pm)_\sigma}&\prod_{\sigma:L_\pi\hookrightarrow\C}H^1(G(F)^+,\cA^{v}(\C))
}
\]
sends $(\phi^\sigma)_\sigma$ to the classes $(\delta^\pm_\sigma(\phi))_\sigma$ defined by the cocycle 
\[
\lambda\longmapsto c_\lambda(\gamma)=\int_{\lambda z_0}^{z_0}\omega^\sigma(\gamma,z)\pm \int_{\lambda z_0}^{z_0}\overline{\omega^\sigma(\gamma,z)},\qquad \lambda\in\Gamma_\gamma.
\]
We can identify $H^1(G(F)^+,\cA^{v}(\C))^{UG(F_{\infty\setminus v})}=\bigoplus_{\gamma\in G(F)^+\backslash G(\A_F^\infty)/U} H^1(\Gamma_\gamma,\C)$ with the singular cohomology of $X_U$, hence the morphism $\delta_\sigma^++\delta_\sigma^-$ provides the comparison isomorphism between singular and deRham cohomologies.

If we assume that the differentials $\omega_\sigma$ are defined over $F$, then we have that $\delta_v^\pm(\phi)=\Omega^{\pm}\tilde\phi^\pm$, where $\phi^\pm\in H^1(G(F)^+,\cA^v(X,\Z))^\pm$ are  modular symbols in the $\pm1$-component under the action of complex conjugation, and $\Omega^\pm\in R_\pi\otimes_\Z\C$. Following the construction of Proposition \ref{prop:latticeTrivial}, the lattice of the abelian variety $A_\Pi$ in this case is $Y=R_\pi+\tau R_\pi$, where $\tau=\Omega^-(\Omega^+)^{-1}$, which is commensurable to $\Omega^+R_\pi+\Omega^-R_\pi\hookrightarrow\Hom(X,\C)$. It is clear that the complex torus $\Hom(X,\C)/\left(\Omega^+R_\pi+\Omega^-R_\pi\right)$ defines $A_\Pi$ by the previous description of $\delta_\sigma^\pm$, hence Conjeture \ref{conj: oda} holds in this case (see \cite[\S 5.3]{Santi}).

By \cite[\S 4.1]{Santi}, the evaluation morphisms
\[
\xymatrix{
H^0(G(F),\cA^{v}(X\otimes D_v,\C))\ar[d]^{\simeq}\ar[r]^{{\rm ev}_v^\pm}&H^0(G(F)^+,\cA^{v}(X\otimes\Delta^0_v,\C))\ar[d]^{\simeq}\\
\prod_{\sigma:L_\pi\hookrightarrow\C}H^0(G(F),\cA^{v}(D_v,\C))\ar[r]^{({\rm ev}_\sigma^\pm)_\sigma}&\prod_{\sigma:L_\pi\hookrightarrow\C}H^0(G(F)^+,\cA^{v}(\Delta^0_v,\C))
}
\]
satisfy
\[
({\rm ev}_\sigma^++{\rm ev}_\sigma^-)(\phi^\sigma)(z_1-z_2)(\gamma)=\int_{z_2}^{z_1}\omega^\sigma(\gamma,z),\qquad z_1,z_2\in\cH.
\]
This implies that a pre-image modulo $Y$ in $H^0(G(F)^+,\cA^v(\Delta_v,A_\Pi(\C)))$ can be chosen to be $(z,g^v)\mapsto\Phi([z,\iota^\infty(g^v)U])$, where $\iota^\infty: G(\A_F^v)\rightarrow G(\A_F^\infty)$ is the natural projection, $[z,\iota^\infty(g^v)U]\in X_U(\C)$ is the corresponding point in the Shimura curve provided by the double coset space description \eqref{Cparam}, and
\[
\Phi:X_U(\C)\longrightarrow {\rm Jac}(X_U)(\C)\longrightarrow A_\Pi(\C),
\]
is the modular parametrization over $F$ given by a suitable multiple of the Hodge class. This defines an element (see \cite[\S 5.3]{Santi})
\[
\Psi\in H^0(G(F)^+,\cA^v(\Delta_v,A_\Pi(\C)\otimes\Q))_\pi;\qquad \Psi(\phi)(z,g^v)=\Phi([z,\iota^\infty(g^v)U]).
\]

Evaluating at $z_K\in \cH$, the unique element fixed by $K^\times$, we obtain
\[
\Psi_K\in H^0(K^\times,\cA^v(A_\Pi(\C)\otimes\Q))_\pi;\qquad\Psi_K(\phi)(g^v)=\Phi([z_K,\iota^\infty(g^v)U]).
\]
\emph{Shimura's reciprocity law} asserts that $[z_K,\iota^\infty(g^v)U]\in X_U(K^{\rm ab})$ and 
\[
[z_K,\iota^\infty(g^v)U]^{\rho_K(k)}=[z_K,k\iota^\infty(g^v)U],\qquad k\in (\A_K^\infty)^\times/K^\times,
\]
where $\rho_K:(\A_K^\infty)^\times/K^\times\rightarrow\Gal(K^{\rm ab}/K)$ is the Artin map.

Notice that, in this case, the homology class $\xi\in H_0(K^\times,C_c^0(T(\A_F^v),\Z))$ is the class of the characteristic function $1_{\cF}\in C_c^0(T(\A_F^v),\Z))$. Thus, for $f=h^v\phi$ with $h^v\in G(\A^v)$ and $\chi\in C(\cG,\bar\Q)$
\[
P(\chi,f)%=\langle\langle \xi\cap\chi,\Psi_K(f)\rangle\rangle
=\int_{\cF}\chi(\rho_K(t))\Psi_K(h^v\phi)(t)dt^\times=\int_{(\A_K^\infty)^\times/K^\times}\chi(\rho_K(t))\Phi([z_K,t\iota^{\infty}(h^v)U])dt^\times.
\]
We conclude that parts $(1)$ and $(2)$ of Conjecture \ref{conjecture_princ} are true in this case and follow from Shimura's reciprocity law. Moreover, part $(3)$ also holds by the Gross--Zagier--Zhang formula \cite[Theorem 1.3.1]{YZZ}.

\subsection{Non-archimedean uniformization of Heegner points}

Let $w\in S_{\rm ns}$ be a finite prime of $F$. By the choice of $S_{\rm ns}$, we have that 
\begin{itemize}
\item $\Pi_{w}$ is Steinberg.

\item $w$ does not split in $K$.
\end{itemize} 

Let $G'$ be the multiplicative group of the quaternion algebra with ramification set $(\Sigma_{\rm ns}\cup S_{\rm ns})\setminus\{w\}$. By the Cerednik-Drinfeld uniformization of the Shimura curve $X_U$, we have that
\begin{equation}\label{p-adicparam}
X_U(K_{w})=G'(F)\backslash \cH_{w}\times G(\A_F^{\infty,w})/U^{w}=\bigsqcup_{\gamma\in G'(F)\backslash G(\A_F^{\infty,w})/U^{w}}\Gamma_\gamma'\backslash\cH_{w};
\end{equation}
where $\cH_{w}=\cH_{w}(K_{w})=\PP^1(K_w)\setminus\PP^1(F_{w})$, $U^{w}:=U\cap G(\A_F^{\infty,w})$ and $\Gamma_\gamma'=\gamma U^{w}\gamma^{-1}\cap G'(F)$.  By \cite[Theorem 4.9]{lennart-felix}, the Jacobian of $X_U$ satisfies
\[
\Jac(X_U)(K_{w})=\bigoplus_{\gamma\in G'(F)\backslash G(\A_F^{\infty,w})/U^{w}}H^1(\Gamma_\gamma',K_{w}^\times)/L_{\Gamma_\gamma'},
\]
where $L_{\Gamma_\gamma'}$ is the image of the map 
\[
H^0(\Gamma_\gamma',\Hom({\rm St}_{w}^\Z,\Z))\longrightarrow H^0(\Gamma_\gamma',\Hom({\rm St}_{w}^{K_{w}^\times},K_{w}^\times))\longrightarrow H^1(\Gamma_\gamma',K_{w}^\times)
\]
given by Remark \ref{rmk:Xint} and the cup product by $c_{{\rm id},w}$. Since 
\begin{eqnarray*}
\bigoplus_{\gamma\in G'(F)\backslash G(\A_F^{\infty,w})/U^{w}}H^1(\Gamma_\gamma',K_{w}^\times)&=&H^1(G'(F),\cA^{w}(K_{v'}^\times))^{U^{w}G(F_\infty)},\mbox{ and}\\
 \bigoplus_{\gamma\in G'(F)\backslash G(\A_F^{\infty,w})/U^{w}}H^0(\Gamma_\gamma',\Hom({\rm St}_{w}^\Z,\Z))&=&H^0(G'(F),\cA^{w}({\rm St}_{w}^\Z,\Z))^{U^{w}G(F_\infty)},
\end{eqnarray*}
we deduce that there is a Hecke equivariant isogeny between $\Jac(X_U)(K_{w})$ and the cokernel of the morphism
\begin{equation}\label{intuniv}
\mint(\cdot)\cup c_{{\rm id},w}:H^0(G'(F),\cA^{w}({\rm St}_{w}^\Z,\Z))^{U^{w}G(F_\infty)}\longrightarrow H^1(G'(F),\cA^{w}(K_{w}^\times))^{U^{w}G(F_\infty)}.
\end{equation}
The automorphic form in $ H^0(G'(F),\cA^{w}({\rm St}_{w}^{\Z},R_\pi))^{U^{w}G'(F_\infty)}$ generating $\pi\mid_{G'(\A_F^{\infty,w})}$, defines $\phi'\in H^0(G'(F),\cA^{w}(X\otimes{\rm St}_{w}^\Z,\Z))^{U^{w}G'(F_\infty)}$. Since $X$ is a free $\Z$-module of rank $[R_\pi:\Z]$, $\phi'$ gives rise to a $\Z$-module in $H^0(G'(F),\cA^{w}({\rm St}_{w}^\Z,\Z))^{U^{w}G'(F_\infty)}$ of rank $[R_\pi:\Z]$. This module is generated by $[R_\pi:\Z]$ harmonic cocycles $\phi_i'\in H^0(G'(F),\cA^{w}({\rm St}_{w}^\Z,\Z))$ obtained from a choice of a basis of $X$. 
The image of this module through \eqref{intuniv} defines the lattice of $A_\Pi$, thus Conjecture \ref{conj: oda} also holds in this setting.

The evaluation morphism ${\rm ev}_{w}$ provides $\psi\in H^0(G'(F),\cA^{w}(X\otimes\Delta_w^0,K_{w}^\times))$ obtained, for a given basis of $X$, by means of the multiplicative integrals
\[
(z_1-z_2)\longmapsto\left(\mint_{z_2}^{z_1}\phi_i'\right)_i
\]
This implies again that a pre-image modulo $Y$ in $H^0(G'(F),\cA^{w}(\Delta_w,A_\Pi(K_{w})))$ can be chosen to be $(z,g^{w})\mapsto\Phi([z,\iota^{\infty}(g^{w})U^{w}])$, where $\iota^{\infty}: G(\A_F^{w})\rightarrow G(\A_F^{\infty,w})$ is the natural projection, and $[z,\iota^\infty(g^w)U^{w}]\in X_U(K_{w})$ is the corresponding point in the Shimura curve provided by the double coset space description \eqref{p-adicparam}.

Evaluating at $z_K\in \cH_{w}$ fixed by $K^\times$ as above, we obtain 
\[
\Psi_K'\in H^0(K^\times,\cA^{w}(A_\Pi(K_{w})\otimes\Q))_\pi;\qquad\Psi_K'(\phi)(g^{w})=\Phi([z_K,\iota^\infty(g^{w})U^{w}]),
\]
whose image is the set of $p$-adic uniformizated Heegner points. %Comparing $p$-adic and complex uniformizations of $X_U$, we deduce that $\Psi_K'=\Psi_K\mid_{G(\A^{v'})}$. Thus,
Parts $(1)$ and $(2)$ of Conjecture \ref{conjecture_princ} are well known in this case (see \cite[Lemma 4.2]{B-D2}), and part $(3)$ follows from the classical Gross--Zagier formula.

\bibliographystyle{halpha}
\bibliography{Anticyclotomic}

\end{document}